\documentclass[a4paper,11pt]{amsart}
\usepackage{amsmath,inputenc,euscript,amssymb,geometry}
\usepackage{graphicx}
\usepackage{graphics}
\usepackage[normalem]{ ulem }
\usepackage{soul}
\usepackage{enumerate}
\usepackage{enumitem}
\usepackage{amssymb}
\usepackage{latexsym}
\usepackage{amssymb,amsbsy,amsmath,amsfonts,amssymb,amscd}
\usepackage{hyperref}
\usepackage{xcolor}

\newcommand{\T}{\mathbb{T}}
\newcommand{\Z}{\mathbb{Z}}
\newcommand{\R}{\mathbb{R}}

\newtheorem{lemma}{Lemma}[section]
\newtheorem{theorem}[lemma]{Theorem}

\newtheorem{rem}[lemma]{Remark}
\newtheorem{remark}[lemma]{Remark}

\begin{document}
\title[]{Blow-up for the 1D cubic NLS}
\author[V. Banica]{Valeria Banica}
    \address[V. Banica]{Sorbonne Universit\'e, CNRS, Universit\'e de Paris, Laboratoire Jacques-Louis Lions (LJLL), F-75005 Paris, France, and Institut Universitaire de France (IUF)} 
\email{Valeria.Banica@sorbonne-universite.fr}

\author[R.Luc\`a]{Renato Luc\`a}
\address[R. Luc\`a]{BCAM - Basque Center for Applied Mathematics, Alameda de Mazarredo 14, E48009 Bilbao, Basque Country - Spain and Ikerbasque, Basque Foundation
for Science, 48011 Bilbao, Basque Country - Spain.}
\email{rluca@bcamath.org}

\author[N. Tzvetkov]{Nikolay Tzvetkov}
    \address[N. Tzvetkov]{Ecole Normale Sup\'erieure de Lyon, UMPA, UMR CNRS-ENSL 5669, 46, all\'ee d'Italie, 69364-Lyon Cedex 07, France} 
\email{nikolay.tzvetkov@ens-lyon.fr}

\author[L. Vega]{Luis Vega}
\address[L. Vega]{BCAM - Basque Center for Applied Mathematics, Alameda de Mazarredo 14, E48009 Bilbao, and Dpto. Matem\'aticas UPV/EHU Apto. 644
48080 Bilbao, Basque Country - Spain} 
\email{luis.vega@ehu.es}
\date\today
\maketitle

\begin{abstract}
We consider the 1D cubic NLS on $\mathbb R$ and prove a blow-up result for functions that are of borderline regularity, i.e. $H^s$ for any $s<-\frac 12$ for the Sobolev scale and $\mathcal F L^\infty$ for the Fourier-Lebesgue scale. This is done by identifying at this regularity a certain functional framework from which solutions exit in finite time. This functional framework allows, after using a pseudo-conformal transformation, to reduce the problem to a large-time study of a periodic Schr\"odinger equation with non-autonomous cubic nonlinearity. The blow-up result corresponds to an asymptotic completeness result for the new equation. We prove it using Bourgain's method and exploiting the oscillatory nature of the coefficients involved in the time-evolution of the Fourier modes. Finally, as an application we exhibit singular solutions of the binormal flow. More precisely, we give conditions on the curvature and the torsion of an initial smooth curve such that the constructed solutions generate several singularities in finite time.\end{abstract}


\section{Introduction}
We consider the Cauchy problem for the focusing 1D cubic NLS on $\mathbb R$: 
\begin{equation}\label{NLS}
iu_t+u_{xx}+ |u|^2u=0,\quad u\vert_{t=0}=u_0.
\end{equation} 
The results in this paper hold also in the defocusing case but for simplicity we will work only with the focusing case.

We recall that with respect to the scaling $\lambda u(\lambda^2 t,\lambda x)$ that leaves invariant the equation, the invariant space on the Sobolev scale is $\dot H^{-\frac 12}$ and the one on the Fourier-Lebesgue scale is $\mathcal F L^\infty$, i.e. Fourier transform in $L^\infty$. 

We shall briefly review the initial value problem results for \eqref{NLS} on Sobolev and Fourier-Lebesgue spaces. The equation is well-posed in $H^s$ for $s\geq 0$  (see \cite{T,GV,CW}) and the flow map is uniformly continuous on bounded sets of $H^s$. It turned out that the last property cannot hold in $H^s$ for $s< 0$ (see \cite{KPV,CCT}). In \cite{KT,KVZ}, it was shown that the Sobolev norms of Schwartz solutions have controlled growth for $-\frac 12< s<0$. In  \cite{HKV}, the authors prove the global well-posedness of \eqref{NLS} in $H^s$ for $s>-1/2$  in the sense that the solution map on the  Schwartz class admits a unique continuous extension to $H^s$, $s>-1/2$. Therefore in  \cite{HKV} the assumption on the initial data is 
\begin{equation}\label{restrr1}
\langle \xi\rangle^{-\frac{1}{2}^+}\widehat{u_0}(\xi)\in L^2.
\end{equation}
The result of  \cite{HKV}  is sharp in the sense that for $s<-1/2$  a norm inflation with loss of regularity appear \cite{CK,Ki,Oh}, in particular \eqref{NLS} is ill-posed in the Hadamard sense in $H^s$, $s<-1/2$. In what concerns the Fourier-Lebesgue spaces, the problem is locally well-posed on $\mathcal F L^p$ for $p<+\infty$  (see \cite{G,VV}). Therefore one can prove local well-posedness of \eqref{NLS} for data satisfying 
\begin{equation}\label{restrr2}
\widehat{u_0}(\xi)\in L^p,\quad \forall\, p<\infty.
\end{equation}
Recall that for $c_0\in\mathbb R$, the Galilean transformation $e^{-ic_0^2 t+ic_0 x}u(t,x-2c_0 t)$ leaves invariant the set of solutions of \eqref{NLS}, as well as the space $\mathcal F L^p$ for all $p\leq \infty$.

In summary, we presently have a well-posedness theory that misses the critical spaces
with respect to scaling. Let us notice that the Dirac mass and its linear Schr\"odinger evolution $e^{it\Delta}\delta_0=\frac{e^{i\frac{x^2}{4t}}}{\sqrt t}$ are of borderline regularity with respect to the scaling and the Galilean symmetries and therefore do not fit in the scope of applicability of the previously described results.
Actually there are simple explicit solutions of (1) for $t > 0$, having the same kind of borderline regularity, given by
$$u_\alpha(t, x) = \alpha e^{i|\alpha|^2\log t} \,\,\frac{e^{i\frac{x^2}{4t}}}{\sqrt t},\quad\alpha\in\mathbb C.$$
In this spirit recently a new insight has been given to solutions of (1) of critical regularity. Solutions which are smooth perturbations of $u_\alpha$ were constructed in \cite{BVCMP, BVJEMS, BVARMA, BVAnnENS,BrV}, solutions of type superpositions of $u_\alpha$ were constructed in \cite{BVAnnPDE, BrV} and smooth perturbations of the latter are given in \cite{Gu}.

Our goal in this work is to prove a suitable form of local well-posedness of \eqref{NLS} for initial data satisfying 
\begin{equation}\label{restrr3}
e^{i \xi^2}\widehat{u_0}(\xi) \mbox { is a } 2\pi\mbox{-periodic function in }L^2(0,2\pi).
\end{equation}
More importantly, the solutions blow-up in finite time due to a loss of information of some phases. Clearly if $u_0\neq 0$ satisfies \eqref{restrr3} then it does not satisfy neither 
\eqref{restrr1} nor \eqref{restrr2}.

Let us next give a  motivation behind the assumption \eqref{restrr3}. It would be more convenient to assume that the initial data is prescribed at time $t=1$ (this can always be achieved thanks to the invariance under translations in time).  Let us notice that $u_\alpha(1,x)$ satisfies \eqref{restrr3} and Theorem \ref{soft_statement} will include $u_\alpha(t,x)$ in the framework of a suitable local well-posedness and blow-up results.
More precisely, here is a soft statement of our first main result. 
\begin{theorem}\label{soft_statement}
Let $s\in (0,1/2)$ and let $u_0$ satisfying \eqref{restrr3}  be such that 
$$
e^{i \xi^2}\widehat{u_0}(\xi)\in H^s(\T)\,,
$$ 
where $\T=\R/ 2\pi\Z$. Then there exists a unique, in a suitable sense, solution of 
$$
iu_t+u_{xx}+ |u|^2u=0
$$
with datum $u_0$ at time $t=1$ in the interval $(0,1]$. The solution $u(t)$ belongs for $t\in(0,1]$ to the functional framework 
$$
e^{it\xi^2}\hat u(t,\xi) \mbox { is }2\pi\mbox{-periodic},
$$
and falls out of this framework at $t=0$.  The dependence with respect to the initial data is continuous in the following sense. 
Let $(u_{0,n})$ be a sequence of initial data satisfying \eqref{restrr3} such that 
$
(e^{i \xi^2}\widehat{u_{0,n}}(\xi)) 
$
converges to 
$
e^{i \xi^2}\widehat{u_0}(\xi)
$
in $H^s(\T)$.  Then for every $t\in (0,1]$ the sequence 
$
(e^{i t\xi^2}\widehat{u_n(t)}(\xi)) 
$
converges to 
$
e^{i t\xi^2}\widehat{u(t)}(\xi)
$
in $H^s(\T)$, where $u_n$ is the solution of \eqref{NLS} with data $u_{0,n}$. 
We have that 
$
u\in C\big((0,1]; {\mathcal S}'(\R)\big)
$
but it blows-up at zero in the sense that 
$
\lim_{t\rightarrow 0^+} u(t,\cdot)
$
does not exist in ${\mathcal S}'(\R)$. 
\end{theorem}
\begin{remark}
The Theorem \ref{soft_statement} extends to $s\geq \frac 12$ by invoking Lemma \ref{lemmawevbis}. Uniqueness will be specified in Theorem \ref{wbl}.
\end{remark}

\begin{remark}
A similar statement holds for data satisfying 
$$
\exists\, \theta\in\R\,\, \colon\,\, e^{i\theta \xi^2}\widehat{u_0}(\xi) \mbox { is }2\pi\mbox{-periodic}
$$
with a different blow-up time (depending on $\theta$).
\end{remark}
Below, we will give a much more precise description of the lack of convergence of $u(t,\cdot)$ as $t\rightarrow 0^+$.

We next continue the motivation of  the assumption \eqref{restrr3}. Observe that for solutions of the free Schr\"odinger equation
$$
iu_t+u_{xx}=0,
$$
it is immediate to see that if $e^{i\xi^2}\widehat u(1, \xi)$ is $2\pi-$periodic then $e^{it\xi^2}\widehat u(t, \xi)$ is also $2\pi-$periodic, for every $t\in\mathbb R$. In fact
$$e^{it\xi^2}\widehat u(t, \xi)=\widehat u_0(\xi)=\sum\limits_{k\in \Z} a_ke^{ik\xi}
$$
for all $\xi$ and for some Fourier coefficients $a_k$. Moreover, any way to measure the size of $\widehat u(t,\xi)$ described in terms of the size of $|a_k|$ will be preserved for all time. It is not so obvious that the periodicity of 
$$
w(t,\xi):=e^{it\xi^2}\widehat u(t, \xi),
$$
holds, at least formally, for all $t$ for solutions of the nonlinear problem  \eqref{NLS}. However let us notice that the equation for $w$ is:
$$
w_t(\eta, t)= \frac{i}{8\pi^3}e^{-it\eta^2}\int e^{it(\xi_1^2-\xi_2^2+(\eta-\xi_1+\xi_2)^2)} w(\xi_1,t)\bar w(\xi_2,t)w(\eta-\xi_1+\xi_2,t) \,d\xi_1d\xi_2.
$$
The phase can be factorized as $(\xi_1-\xi_2)(\xi_1-\eta)$, so it is invariant under translations. Thus the equation of $w$ is compatible with periodicity; the solutions we shall construct are in this framework. 
Now we shall give a version of Theorem \ref{soft_statement} that takes in account the physical variable point of view. 
\begin{theorem}\label{wbl}
For initial data 
$$u_0(x)=e^{i\frac{x^2}4}f(x),$$
at time $t=1$, with $f$ a periodic function in $H^s(0,4\pi)$ and $s\in (0,\frac 12)$ there exists a unique solution of the 1D cubic NLS equation \eqref{NLS} that blows up in finite time in the following sense: the solution belongs for $t\in(0,1]$ to the functional framework 
\begin{equation}\label{ansatzth}
w(t,\xi):=e^{it\xi^2}\hat u(t,\xi) \mbox { is }2\pi\mbox{-periodic},
\end{equation}
$w\in C\big((0,1]; H^s(\T)\big),$
and falls out of this framework at $t=0$. More precisely, by denoting $A_j(t)$ the Fourier coefficients of $w(t)$, then 
$$\sup_{t\in(0,1]}\|\{A_j(t)\}_{j\in\mathbb Z}\|_{l^{2.s}}\leq C(\|f\|_{H^s(0,4\pi)}),$$
and there exists a sequence $\{\alpha_j\}_{j\in\mathbb Z}\in l^{2,s}$ such that
\begin{equation}\label{blupth}
|A_j(t)-e^{ i(|\alpha_j|^2-2\|f\|_{L^2(0,4\pi)}^2)\log t}\alpha_j|\leq C(\|f\|_{H^s(0,4\pi)})\,t, \quad \forall j\in\mathbb Z, t\in(0,1).
\end{equation}
Finally the solution can be written as
$$u(t,x)=\sum_{j\in\mathbb Z} A_j(t)\frac{e^{i\frac{(x-j)^2}{4t}}}{\sqrt{t}}=\sum_{j\in\mathbb Z} A_j(t)e^{it\Delta}\delta_j,$$
where $\delta_j$ is the Dirac mass at $j$, so in particular $u$ focalises at $t=0$.
\end{theorem}
\begin{remark}
As a consequence of the above result even though $|A_j(t)|$ have a limit for all $j$ when $t$ goes to zero, the Fourier coefficient themselves don't, due to the the phase loss $(|\alpha_j|^2-2\|f\|_{L^2(0,4\pi)}^2)\log t$. This is what we refer as a phase blow up. Loss of phase phenomena for nonlinear Schr\"odinger equation have been noticed in \cite{M,MRS}.
\end{remark}

The proof of Theorem \ref{wbl} goes as follows. Using the pseudo-conformal transformation
\begin{equation}\label{pseudo}
v(\tau,y):=\frac{e^{i\frac{y^2}{4\tau}}}{\sqrt {\tau}}\,\overline u(\frac 1\tau,\frac y{\tau}),
\end{equation}
and imposing the ansatz \eqref{ansatzth} reduces the problem to constructing solutions $v(\tau)$ for $\tau\geq 1$ for a $2\pi$-periodic Schr\"odinger equation with a non-autonomous cubic nonlinearity:
\begin{equation}\label{NLSvtauy}
iv_\tau+v_{yy}+ \frac 1{\tau}|v|^2v=0.
\end{equation} 
These solutions are obtained by gluing solutions on time-intervals of size one constructed in turn by using Bourgain's arguments \cite{B94}. In a next step we start the study of their asymptotic behaviour as $\tau$ goes to infinity. To do so we first derive pointwise estimates for the Fourier in time transform of the derivative in time of the Fourier modes of $v(\tau)$. Then, we show that the $H^s$-norms of $v(\tau)$ are bounded in time by using also strongly the oscillatory nature of the coefficients involved in the time-evolution of the Fourier modes of $v(\tau)$. Finally the limit of the modulations of the Fourier modes of $v(\tau)$, that yields the blow-up phenomena of Theorem \ref{wbl}, is obtained by gathering all this previous information and using again the oscillatory nature of the time-evolutions of the Fourier modes of $v(\tau)$.

Existence of solutions $u$ as in Theorem \ref{wbl} was proved in \cite{BVAnnPDE} for more regular data, and correspond to existence of wave operators for the equation on $v$. More precisely given $\{\alpha_k\}_{k\in\mathbb Z}\in l^{2,s}$ with $s>\frac 12$ solutions for \eqref{NLS} of type \eqref{ansatzth} and satisfying \eqref{blupth} were obtained, but without any information on the set of functions $f$ reached at $t=1$. Regular perturbations of such solutions were constructed in \cite{Gu}. Also, in \cite{BrV} the case when $\{\alpha_k\}_{k\in\mathbb Z}\in l^p$ was treated, and local in time small solutions were obtained for $1<p<\infty$, with an improvement of the time of existence under smallness in $l^\infty$, and global small solutions were obtained for $p=2$. In turn, Theorem \ref{wbl} is an asymptotic completeness result for the equation on $v$, at regularity $ l^{2,s}$ with  $0<s<\frac 12$, i.e. given any $f\in H^s$ we construct a solution of \eqref{NLS} of type \eqref{ansatzth} and satisfying \eqref{blupth}. 

Finally, based on Theorem \ref{wbl} we prove a result for the binormal flow:
$$\chi_t(t,x)=\chi_x\times\chi_{xx}(t,x),$$
which is used as a model for the dynamics in time in a 3D fluid of superfluid of a filament of vorticity concentrated on a 3D-curve $\chi(t)$ parametrized by arclenght parameter $x$ (\cite{DaR}, \cite{JS} and the references therein). The binormal flow is rigorously connected to the focusing equation \eqref{NLS} due to Hasimoto's transform \cite{H}. The self-similar solutions of the binormal flow, that are smooth curves generating a corner in finite time, correspond to phenomena that appear in fluids and superfluids (Figures 1 and 2), and were studied rigorously in \cite{GRV}. 
We recall that in \cite{BVAnnPDE} it is constructed the evolution in time, according to the binormal flow, of curves that are initially  polygonal lines and that become instantaneously smooth. As a consequence, since the binormal flow is time reversible, smooth curves can generate several singularities in the shape of corners in finite time. This result is based on the existence of wave operators for the Schr\"odinger equation \eqref{NLSvtauy}. Here, we shall use the asymptotic completeness result for \eqref{NLSvtauy}, stemming from Theorem \ref{wbl}, to obtain a general criterium for generating singularities in finite time through the binormal flow. 

\begin{theorem}\label{thbf}
Let $\chi_1$ be a curve with $4\pi-$periodic curvature and torsion equal to $\frac x{2}$ modulo a $H^{\frac 32^+} 4\pi$-periodic function. Then there exists $\chi(t)$, with $\chi(1)=\chi_1$, a strong solution of the binormal flow on $\mathbb R^*$ and weak solution on $\mathbb R$ which generates several corner-singularities at time $t=0$.
\end{theorem}
We first note that the hypothesis on the data $\chi_1$ is satisfied by the family of self-similar solutions of the binormal flow, whose profile curvature is constant and whose torsion equal to $\frac x2$. In fluids vortex filaments dynamics having the behavior of the self-similar solutions is observed for delta-wing planes, see Figure 1. In superfluids this family of solutions were used by physicists studying the vortex lines reconnections and the Kelvin waves that emerge from, see Figure 2. We note also that if we want a blow-up after time $t_0$ then we should start with a curve with $4\pi t_0$-periodic curvature and torsion equal to $\frac x{2t_0}$ modulo a $H^{\frac 32^+} 4\pi t_0$-periodic function.

\begin{center}$\includegraphics[width=2in]{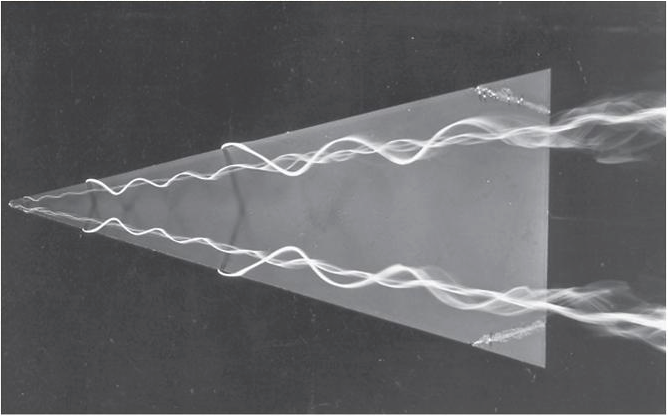}$\end{center}
{\footnotesize{Figure 1. Vortex filaments in a fluid encountering a triangular obstacle of the delta wing type, H. Werl\'e, ONERA 63.\\
}}\bigskip
\begin{center}
$\includegraphics[width=4in]{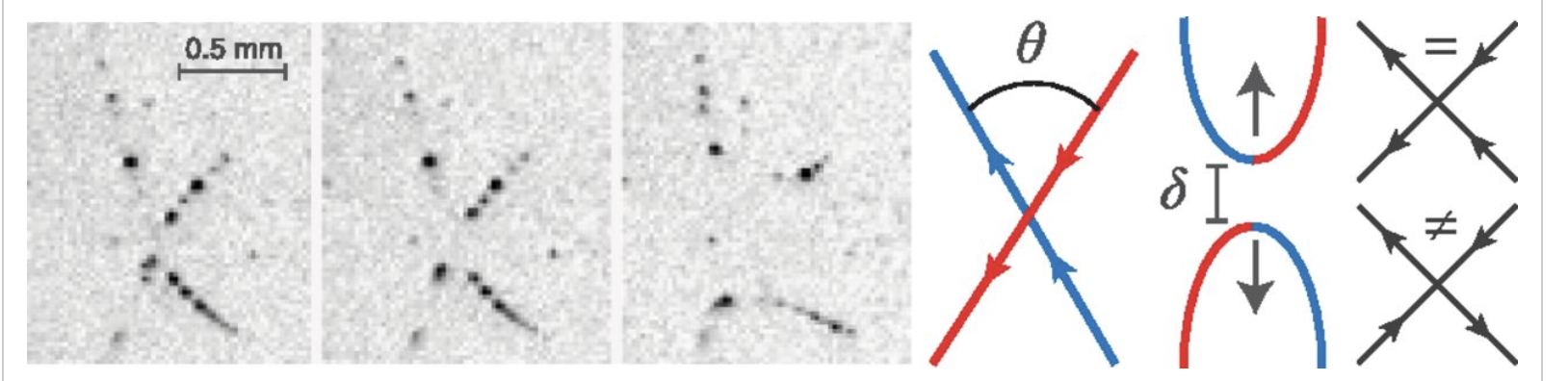}$ \end{center}
{\footnotesize{
Figure 2. Reconnection scaling in quantum fluids, E. Fonda, K. Sreenivasan and D. Lathrop, PNAS 19}}\\

The proof of Theorem \ref{thbf} relies first on a control in time of the $l^{2,\frac 32^+}$-norms of the sequences $\{A_j(t)\}_{j\in\mathbb Z}$ in Theorem \ref{wbl} provided that $f\in H^{\frac 32^+}$, and then on the proof of the persistence of this regularity in the limit sequence $\{\alpha_j\}_{j\in\mathbb Z}$. Finally, the construction in \cite{BVAnnPDE} for $\{\alpha_j\}_{j\in\mathbb Z}\in l^{2,\frac 32^+}$ allows to construct the solutions  in Theorem \ref{thbf}.

The paper is organized as follows: in the next section we give the proof of Theorem \ref{wbl} while the last section contains the proof of Theorem \ref{thbf}. The proof of Theorem \ref{soft_statement} is contained in the one of Theorem \ref{wbl}.

{\bf{Acknowledgements:}}  This research is partially supported as follows. VB is partially supported by the Institut Universitaire de France, by the French ANR project SingFlows. RL is supported by BERC program 2022-2025 and by MICINN (Spain) projects Severo Ochoa CEX2021-001142, PID2021-123034NB-I00 and by the Ramon y Cajal fellowship RYC2021-031981-I. NT is partially supported by the ANR project Smooth ANR-22-CE40-0017. LV is funded by MICINN (Spain) projects Severo Ochoa CEX2021-001142, and PID2021-126813NB-I00 (ERDF A way of making Europe), and by Eusko Jaurlaritza project IT1615-22 and BERC program. The authors would like to thank the referees for their careful reading and suggestions. 
\section{Proof of Theorem \ref{wbl}}
\subsection{The equations} 
To prove Theorem \ref{wbl} we look thus for solutions of \eqref{NLS} on $t>0$ of the type
\begin{equation}\label{ansatz}
u(t,x)=\sum_jA_j(t)\frac{e^{i\frac{(x-j)^2}{4t}}}{\sqrt{t}},
\end{equation}
with initial data 
\begin{equation}\label{id}
u(1,x)=e^{i\frac{x^2}{4}}f(x)=\sum_jA_j(1)e^{i\frac{(x-j)^2}{4}} \iff A_j(1)=\hat f(j)e^{-i\frac{j^2}{4}}.
\end{equation}

After pseudo-conformal transformation \eqref{pseudo} such a solution $u(t)$ on $0<t\leq 1$ turns into a solution $v(t)$ of the NLS on $1\leq t<\infty$:
\begin{equation}\label{NLSv}
iv_t+v_{xx}+\frac 1{t}|v|^2v=0,
\end{equation} 
 that is a $2\pi-$periodic function (modulo constants, that we avoid for the sake of the clearness of the presentation, and that do not affect the arguments and results) of the type:
\begin{equation}\label{vB}
v(t,x)=\sum_j B_j(t) e^{itj^2}e^{ixj},
\end{equation} 
where
\begin{equation}\label{ABrelation}
B_j(t)=\overline{A_j}(\frac 1t).
\end{equation} 
In particular the Fourier coefficients of $v(t)$ satisfy 
$$|\hat v(t,j)|=|A_j(1/t)|,$$ 
thus estimating Sobolev norms of $v(t)$ is equivalent to weighted estimates on the sequence $A_j(1/t)$. 
\begin{rem}
In particular from Theorem \ref{wbl} we obtain periodic NLS solutions $v$ with Fourier modes converging in modulus:
$$||\hat v(t,j)|^2-|\alpha_j|^2|\leq \frac Ct,$$
but with fluctuations of the modes.
\end{rem}

Also, we can note that as usually for questions concerning the long-time behavior of cubic NLS, we have the relation:
$$B_j(t)=\mathcal F(e^{-it\Delta}v(t))(j).$$\\
We denote then
$$B(t,x):=\sum_j B_j(t)e^{ijx}=e^{-it\Delta}v(t)(x),$$
and noting that we have the following expression for the Fourier transform in both time and space variable:
$$v(t,x)=\sum_j(e^{itj^2}B_j(t))e^{ijx}\Longrightarrow \hat v(\lambda,j)=\hat B_j(\lambda+j^2),$$
we introduce Bourgain's norms:
\begin{equation}\label{B}\|v\|_{X^{s,b}}:=\Big(\int \sum_j \langle j\rangle^{2s} \langle \lambda+j^2\rangle^{2b} |\widehat{v}(\lambda,j)|^2d\lambda\Big)^\frac 12=\Big(\int \sum_j \langle j\rangle^{2s} \langle \lambda\rangle^{2b} |\widehat{B}_{j}(\lambda)|^2d\lambda\Big)^\frac 12=:\|B\|_{H^{s,b}},\end{equation} 
and for localizations in time $f_\nu(t)= \eta_\nu(t) f(t)$ in $(\nu-1,\nu+2)$, valued $1$ on $[\nu,\nu+1]$:
\begin{equation}\label{Bloc}\|v\|_{X_\nu^{s,b}}:=\|v_\nu \|_{X^{s,b}}=\Big(\int \sum_j \langle j\rangle^{2s} \langle \lambda\rangle^{2b} |\widehat{B}_{j,\nu}(\lambda)|^2d\lambda\Big)^\frac 12=:\|B\|_{H^{s,b}_\nu}.
\end{equation}


Now we remark that $u(t)$ of ansatz \eqref{ansatz} is a solution of \eqref{NLS} if and only if the sequence $A(t)=\{A_j(t)\}_{j\in\mathbb Z}$ is solution of the system
$$i\partial_t A_k(t)=\frac{1}{ t}\sum_{(j_1,j_2,j_3)\in \Sigma_k}e^{-i\frac{k^2-j_1^2+j_2^2-j_3^2}{4t}}A_{j_1}(t)\overline{A_{j_2}(t)}A_{j_3}(t),$$
translating into the sequence $B(t)=\{B_j(t)\}_{j\in\mathbb Z}$ being a solution of the system (modulo constants)
\begin{equation}\label{Bjsyst}i\partial_t B_k(t)=\frac{1}{ t}\sum_{(j_1,j_2,j_3)\in \Sigma_k}e^{-it(k^2-j_1^2+j_2^2-j_3^2)}B_{j_1}(t)\overline{B_{j_2}(t)}B_{j_3}(t),\end{equation}
where
$$\Sigma_k:=\{(j_1,j_2,j_3),\, k-j_1+j_2-j_3=0\}.$$
Let us introduce the nonresonant set
$$NR_k=\{(j_1,j_2,j_3),\, k-j_1+j_2-j_3=0,\, k^2-j_1^2+j_2^2-j_3^2\neq 0\}.$$
First we note that if $k-j_1+j_2-j_3=0$ we have 
$$\omega_{k,j_1,j_2}:=k^2-j_1^2+j_2^2-j_3^2=(k-j_1)(j_1-j_2).$$ 
For any real function $a$, using a standard symmetrization argument, as $a(k)-a(j_1)+a(j_2)-a(j_3)$ vanishes on the resonant set, we have:
\begin{multline}\label{Bjcons}\partial_t \sum_k a(k)|B_k(t)|^2\\
=\frac{1}{2 ti}\sum_{k-j_1+j_2-j_3=0}(a(k)-a(j_1)+a(j_2)-a(j_3))e^{-it\omega_{k,j_1,j_2}}B_{j_1}(t)\overline{B_{j_2}(t)}B_{j_3}(t)\overline{B_k(t)}\\
=\frac{1}{2ti}\sum_{k;NR_k}(a(k)-a(j_1)+a(j_2)-a(j_3))e^{-it\omega_{k,j_1,j_2}}B_{j_1}(t)\overline{B_{j_2}(t)}B_{j_3}(t)\overline{B_k(t)}.
\end{multline}

Therefore by taking $a\equiv1$ we obtain that the system conserves the ``mass" :
\begin{equation}\label{mass}\sum_k|B_k(t)|^2=\|v(1)\|_{L^2(0,2\pi)}^2=\|f\|_{L^2(0,2\pi)}^2=:M.\end{equation}

In particular \eqref{Bjsyst} can be written as
\begin{equation}\label{Bjsystdev}i\partial_t B_k(t)=\frac{1}{ t}\sum_{NR_k}e^{-it\omega_{k,j_1,j_2}}B_{j_1}(t)\overline{B_{j_2}(t)}B_{j_3}(t)+\frac 1t\left(2M-|B_k(t)|^2\right)B_k(t).\end{equation}
At this point we are in a situation to apply a modified scattering technique for 1D cubic NLS as initiated by Ozawa (\cite{Oz}), Hayashi and Naumkin (\cite{HaNa}) and Carles (\cite{Ca}). More precisely, we will show that the first term in the right-hand-side of \eqref{Bjsystdev} is controllable by oscillatory integral techniques, while the second term will dictate the long time dynamics of \eqref{Bjsystdev}. Let us notice that in the proof of modified scattering for 1D cubic NLS in \cite{Oz} (see (2.4) in the revisited proof in \cite{KaPu}) the modulus of the perturbative term is absolutely integrable in time, while in here the perturbative term can be shown to be integrable in time only in the sense of oscillatory integrals.

\subsection{Construction \`a la Bourgain of the solution $v(t)$ on $[1,\infty)$}\label{ssectconstr}
In view of \eqref{id}, \eqref{vB} and \eqref{ABrelation}, we consider the initial condition 
$$b_j:=\overline{A_j}(1)=\overline{\hat f(j)e^{-ij^2}}.$$ 
Let
$$\eta\in\mathcal C^\infty_0(-1,2),\quad \eta(t)=1,\quad \forall t\in [0,1],$$
such that we have a partition of unity 
$$\frac12\sum_{\nu\in\mathbb N} \eta_\nu(t)=1,\quad \forall t\in [1,\infty),$$
with
$$\eta_\nu(t)=\eta (t-\nu).$$

For all $\nu\in\mathbb N^*$ we construct $\{B_{k}^\nu\}_{k\in\mathbb Z}$ recursively as follows, by using at each step Bourgain's method (\cite{B94}, see also \S 3.5.1 in \cite{ErTz}). First we construct $\{B_{k}^1\}_{k\in\mathbb Z}$ solution of:
$$i\partial_t B_{k}^1=\frac 1t \eta_1(t)\sum_{(j_1,j_2,j_3)\in \Sigma_k}e^{-it\omega_{k,j_1,j_2}}B_{j_1}^1\overline{B_{j_2}^1}B_{j_3}^1(t),\quad B_{k}^1(1)=b_k.$$
We note that $\{B_{k}^1\}_{k\in\mathbb Z}$ solves on $[1,2]$ the equation we are interested in:
$$
i\partial_t B_{k}^1=\frac 1t \sum_{(j_1,j_2,j_3)\in \Sigma_k}e^{-it\omega_{k,j_1,j_2}}B_{j_1}^1\overline{B_{j_2}^1}B_{j_3}^1(t).
$$
Next we construct $\{B_{k}^2\}_{k\in\mathbb Z}$ solution of 
$$i\partial_t B_{k}^2=\frac 1t \eta_2(t)\sum_{(j_1,j_2,j_3)\in \Sigma_k}e^{-it\omega_{k,j_1,j_2}}B_{j_1}^2\overline{B_{j_2}^2}B_{j_3}^2(t),\quad B_{k}^2(2)=B_{k}^1(2).$$
And so on. Then
$$B_k(t):=\left\{\begin{array}{c}B_{k}^1(t),\forall t\in [1,2],\\B_{k}^2(t),\forall t\in [2,3],\\\mbox{etc}.\end{array}\right.$$
solves \eqref{Bjsystdev} on $(1,\infty)$ with $B_k(1)=b_k$, therefore $v$ defined as in \eqref{vB} solves \eqref{NLSv} on $(0,1)$ with $v(1,x)=\sum_j b_j e^{ij^2}e^{ixj}$. In view of the choice of $b_j's$ we recover the initial data $v(1)=\overline{f}$. Therefore we have obtained solutions $u(t)$ on $t\in(0,1]$ of the type \eqref{ansatz}, with initial data \eqref{id}.\\

Also, arguing as in  \S 3.5.1 in \cite{ErTz} we have the following informations involving Bourgain's norms \eqref{B}-\eqref{Bloc}, for $0\leq s, \frac 12<b<\frac 58, \nu\in\mathbb N, t\in[\nu,\nu+1]$:
\begin{equation}\label{Best}
\|\{B_j(t)\}\|_{l^{2,s}}=\|v(t)\|_{H^s}\leq C\|v\|_{X^{s,b}_\nu}=C\|B\|_{H^{s,b}_\nu}\leq C \|v(\nu)\|_{H^s}=C\|\{B_j(\nu)\}\|_{l^{2,s}},
\end{equation}
where the constant $C$ is independent of $t$, $\nu$ and depends only of the $l^2$ norm of $\{b_k\}$.

\subsection{Extra-properties on $B_j$'s}
In order to determinate the large time behaviour of $B_j(t)$ we shall first obtain some large time controls. We start with a decay estimate of the Fourier transform in time of a localization of $\partial_tB_j(t)$.

\begin{lemma}\label{bounddert}
For any $\nu\in\mathbb N$, and $\eta_\nu$ a smooth cutoff supported in $[\nu,\nu+1]$, we have for any  $\epsilon>0$ and $\epsilon'>0$ the following estimate on the Fourier transform in time:
$$ |\mathcal F(\eta_\nu(\cdot)\partial_t B_k(\cdot))(\lambda)|\leq \frac {C(\epsilon',\|\{B_j(1)\}\|_{ l^{2}})\||\xi|^{1^+}\hat\eta_\nu(\xi)\|_{L^\infty}}{\nu\,\langle \lambda\rangle^{\min\{\frac \epsilon 2,(b-\frac 12)^-\}}} \|\{B_j(\nu)\}\|_{l^{2,\epsilon+2\epsilon'}},\quad\forall k\in\mathbb Z,\lambda\in\mathbb R,$$
\end{lemma}
\begin{proof}
We denote
$$\Lambda_{k,m}=\{(j_1,j_2)\in\mathbb Z^2, (k-j_1)(j_1-j_2)=m\},$$
and we have:
\begin{equation}\label{Fpartialtest}\mathcal F(\eta_\nu(\cdot)\partial_t B_k(\cdot))(\lambda)\end{equation}
$$=\int e^{it\lambda}\eta_\nu(t)\frac 1t\sum_{m\neq 0}\sum_{j_1,j_2\in\Lambda_{k,m}}e^{-itm}(B_{j_1}\overline{B}_{j_2}B_{k-j_1+j_2})(t)dt$$
$$+\int e^{it\lambda}\eta_\nu(t)\frac 1t\left(2M-|B_k(t)|^2\right)B_k(t)dt.$$
In the following we shall estimate the first term\footnote{This term is controllable directly for $\{B_j\}\in l^{2,s}$ with $s>\frac 12$, but not for lower regularity.} , that we denote by $F(\nu,k,\lambda)$:
$$\frac 1\nu \int \int \sum_{m\neq 0} e^{-it(\lambda-\lambda_1+\lambda_2-\lambda_3+m)}\tilde\eta_\nu(t)dt \sum_{j_1,j_2\in\Lambda_{k,m}}\hat B_{j_1,\nu}(\lambda_1)\widehat {\overline{B}}_{j_2,\nu}(\lambda_2)\hat B_{k-j_1+j_2,\nu}(\lambda_3) d\lambda_1d\lambda_2d\lambda_3$$
$$=\frac 1\nu \int\sum_{m\neq 0}\widehat{\tilde \eta}_\nu(\lambda-\lambda_1+\lambda_2-\lambda_3+m) \sum_{j_1,j_2\in\Lambda_{k,m}}\hat B_{j_1,\nu}(\lambda_1)\widehat{\overline{ B}}_{j_2,\nu}(\lambda_2)\hat B_{k-j_1+j_2,\nu}(\lambda_3)d\lambda_1d\lambda_2d\lambda_3,$$
where $\tilde \eta_\nu(t)=\frac \nu t\eta_{\nu}(t)$ and preserve the same estimates as $\eta_\nu$ at Fourier level (see \cite{BrV}). 
We perform Cauchy-Schwarz in $j_1,j_2$ and use the divisor bounds in $\mathbb Z$ (which is an Euclidean division domain), for any $\epsilon'>0$:
$$\# \Lambda_{k,m}\leq C(\epsilon')|m|^{\epsilon'}=C(\epsilon')|(k-j_1)(j_1-j_2)|^{\epsilon'},\quad\forall j_1,j_2\in\Lambda_{k,m},$$ 
to get
$$|F(\nu,k,\lambda)|\leq \frac {C({\epsilon'})}\nu \int\sum_{m\neq 0}|\widehat{\tilde \eta}_\nu(\lambda-\lambda_1+\lambda_2-\lambda_3+m)| $$
$$\times\Big(\sum_{j_1,j_2\in\Lambda_{k,m}}|(k-j_1)(j_1-j_2)|^{2{\epsilon'}}|\hat B_{j_1,\nu}(\lambda_1)\widehat{\overline{ B}}_{j_2,\nu}(\lambda_2)\hat B_{k-j_1+j_2,\nu}(\lambda_3)|^2\Big)^\frac 12d\lambda_1d\lambda_2d\lambda_3.$$
Now we perform Cauchy-Schwarz in $m$, for $N\in \mathbb N$:
$$|F(\nu,k,\lambda)|\leq\frac {C(\epsilon')}\nu \int\Big( \sum_{m\neq 0}|\widehat{\tilde \eta}_\nu(\lambda-\lambda_1+\lambda_2-\lambda_3+m)|^2\langle \lambda-\lambda_1+\lambda_2-\lambda_3+m\rangle ^N \Big)^\frac 12$$
$$\times \Big(\sum_{m\neq 0,j_1,j_2\in\Lambda_{k,m}}\frac{|(k-j_1)(j_1-j_2)|^{2\epsilon'}}{\langle \lambda-\lambda_1+\lambda_2-\lambda_3+m\rangle ^N}|\hat B_{j_1,\nu}(\lambda_1)\widehat{\overline{ B}}_{j_2,\nu}(\lambda_2)\hat B_{k-j_1+j_2,\nu}(\lambda_3)|^2\Big)^\frac 12d\lambda_1d\lambda_2d\lambda_3$$
$$\leq  \frac {C(\epsilon')\||\xi|^{\frac {(N+1)^+}2}\hat\eta_\nu(\xi)\|_{L^\infty}}\nu$$
$$\times \int\Big(\sum_{m\neq 0,j_1,j_2\in\Lambda_{k,m}}\frac{|(k-j_1)(j_1-j_2)|^{2\epsilon'}}{\langle \lambda-\lambda_1+\lambda_2-\lambda_3+m\rangle ^N}|\hat B_{j_1,\nu}(\lambda_1)\widehat{\overline{ B}}_{j_2,\nu}(\lambda_2)\hat B_{k-j_1+j_2,\nu}(\lambda_3)|^2\Big)^\frac 12d\lambda_1d\lambda_2d\lambda_3.$$
Eventually we perform Cauchy-Schwarz in $\lambda_1,\lambda_2,\lambda_3$ to get using $b>\frac 12$:
$$|F(\nu,k,\lambda)|\leq \frac {C(\epsilon')\||\xi|^{\frac {(N+1)^+}2}\hat\eta_\nu(\xi)\|_{L^\infty}}\nu \Big(\int \sum_{m\neq 0, j_1,j_2\in\Lambda_{k,m}}\frac{|(k-j_1)(j_1-j_2)|^{2\epsilon'}}{\langle \lambda-\lambda_1+\lambda_2-\lambda_3+(k-j_1)(j_1-j_2)\rangle ^N}\times$$
$$\times   \langle\lambda_1\rangle^{1^+}|\hat B_{j_1,\nu}(\lambda_1)|^2  \langle\lambda_2\rangle^{1^+} |\hat B_{j_2,\nu}(\lambda_2)|^2 \langle\lambda_3\rangle^{1^+} |\hat B_{k-j_1+j_2,\nu}(\lambda_3)|^2d\lambda_1d\lambda_2d\lambda_3\Big)^\frac 12.$$
Even without taking advantage of the denominator, we get the decay in $\nu$:
$$|F(\nu,k,\lambda)|\leq \frac {C(\epsilon')\||\xi|^{\frac {1}2^+}\hat\eta_\nu(\xi)\|_{L^\infty}}\nu \|B\|_{H^{0,\frac 12^+}_\nu}^2\|B\|_{H^{2\epsilon',\frac 12^+}_\nu} $$
$$\leq \frac {C(\epsilon')\||\xi|^{\frac {1}2^+}\hat\eta_\nu(\xi)\|_{L^\infty}}\nu C(\|\{B_j(1)\}\|_{ l^{2}})\|\{B_j(\nu)\}\|_{l^{2,2\epsilon'}},$$
where in the last line we have used estimate \eqref{Best}, the ``mass" conservation \eqref{mass} and 
\begin{equation}\label{estepsilon}|(k-j_1)(j_1-j_2)|^{2\epsilon'}\leq \max\{\langle j_1\rangle, \langle j_2\rangle,\langle k-j_1+j_2\rangle\}^{4\epsilon'}.\end{equation}
To get the decay in $\lambda$ we argue as follows. By taking $N=1$ (so the upper-bound will involve $\||\xi|^{1^+}\hat\eta_\nu(\xi)\|_{L^\infty})$ we get directly $\langle \lambda\rangle^{-\frac 12}$ except in the case 
$$|\lambda-(\lambda_1-\lambda_2+\lambda_3-(k-j_1)(j_1-j_2))|\leq\frac{|\lambda|}2.$$ 
If $|\lambda_1-\lambda_2+\lambda_3|\geq \frac{|\lambda|}4$ then $|\lambda_j|\geq \frac{|\lambda|}{12}$ at least for one $j\in\{1,2,3\}$, and in this case we get $\frac 1{\langle\lambda\rangle^{(b-\frac 12)^-}}$ decay by using the $H_\nu^{0,b}$ and $H_\nu^{2\epsilon',b}$ norms. Eventually, we are left with the case 
$$|\lambda+(k-j_1)(j_1-j_2)|\leq\frac 34|\lambda|\Longrightarrow |(k-j_1)(j_1-j_2)|\geq \frac 14|\lambda|.$$ 
Then we get $\frac 1{\langle\lambda\rangle^{\frac{\epsilon}2}}$ decay with a $\|\{B_j(\nu)\}\|_{l^{2,\epsilon+2\epsilon'}}$ bound by using \eqref{estepsilon} with $\epsilon$ instead of $2\epsilon'$.

We note here that the second term in \eqref{Fpartialtest} can be treated by one integration by parts which yields the factor $\frac 1{\nu \lambda}$, then in a similarly way as above for $F(\nu,k,\lambda)$. As $(b-\frac 12)^-<1$ the analysis of the second term  in \eqref{Fpartialtest} is completed and the Lemma follows. 

\end{proof}
\begin{rem}We do not get an upper-bound on $\mathcal F(\eta_\nu(t)\partial_t B_k(t))(\lambda)$ by working with $B_j$'s just in $l^2$.
\end{rem}

Now we will obtain a control in time the weighted norms of $\{B_j(t)\}$ as follows.
\begin{lemma}\label{lemmawev}
For any $t\geq 1, r \in[0,\frac 12)$ we have:
$$\|\{B_j(t)\}\|_{l^{2,r}}\leq C( r,\|\{B_j(1)\}\|_{ l^{2,r}}).$$
\end{lemma}

\begin{proof}

The case $r=0$ follows by mass conservation \eqref{mass}. By using \eqref{Bjcons} and an integration by parts we have
\begin{equation}\label{diffr}\sum_k \langle k\rangle^{2r}|B_k(t)|^2-\sum_k \langle k\rangle^{2r}|B_k(1)|^2\end{equation}
$$=\int_1^t \sum_{k;NR_k}(\langle k\rangle ^{2r}-\langle j_1\rangle ^{2r}+\langle j_2\rangle ^{2r}-\langle j_3\rangle ^{2r})e^{-i\tau (k^2-j_1^2+j_2^2-j_3^2)}B_{j_1}\overline{B_{j_2}}B_{j_3}\overline{B_{k}}(\tau)\frac{d\tau}{\tau}$$
$$\leq C\sum_{k;NR_k}|\varphi_{k,j_1,j_2,j_3}B_{j_1}\overline{B_{j_2}}B_{j_3}\overline{B_{k}}(t)\frac 1t|+C\sum_{k;NR_k}|\varphi_{k,j_1,j_2,j_3}B_{j_1}\overline{B_{j_2}}B_{j_3}\overline{B_{k}}(1)|$$
$$+C\int_1^t \sum_{k;NR_k}|\varphi_{k,j_1,j_2,j_3}B_{j_1}\overline{B_{j_2}}B_{j_3}\overline{B_{k}}(\tau)|\frac{d\tau}{\tau^2}$$
$$+C\Big|\int_1^t \sum_{k;NR_k}\varphi_{k,j_1,j_2,j_3}e^{i\tau  (j_1^2-j_2^2+j_3^2-k^2)}B_{j_1}\overline{B_{j_2}}B_{j_3}\overline{\partial_{\tau} B_{k}}(\tau)\frac{d\tau}{\tau}\Big|,$$
where 
$$\varphi_{k,j_1,j_2,j_3}:=\frac{\langle k\rangle ^{2r}-\langle j_1\rangle ^{2r}+\langle j_2\rangle ^{2r}-\langle j_3\rangle ^{2r}}{k^2-j_1^2+j_2^2-j_3^2}.$$
We aim to prove that  for any $t\geq 1$ and $r\in (0,\frac 12)$ we have:
\begin{equation}\label{neweq}\sum_k \langle k\rangle^{2r}|B_k(t)|^2\leq \sum_k \langle k\rangle^{2r}|B_k(1)|^2+ C(r,\|\{B_j(1)\}\|_{l^2})\sup_{\tau\in[1,t]}\|\{B_j(\tau)\}\|_{l^{2,r}}^{0^+},\end{equation}
so by resorbing the last terms in the left-hand-side we obtain the Lemma for $r\in (0,\frac 12)$:
$$\sum_k \langle k\rangle^{2r}|B_k(t)|^2\leq C(r,\|\{B_j(1)\}\|_{l^2})\sum_k \langle k\rangle^{2r}|B_k(1)|^2.$$

To get \eqref{neweq} we shall use the following technical lemma (estimates in the cases $r\geq \frac 12$ were obtained in Lemma 2.5 in \cite{FoSe}).

\begin{lemma}\label{lemmabound}
For $r<\frac 12$ we have
\begin{equation}\label{phibound}
 |\varphi_{k,j_1,j_2,j_3}|\leq \frac C{\max\{\langle k\rangle,\langle j_1\rangle,\langle j_2\rangle,\langle j_3\rangle\}},\quad \forall  k, (j_1,j_2,j_3)\in NR_k.
\end{equation}
As a consequence, for two sequences $\{N_j\}$ and $\{P_j\}$ of positive numbers we have
\begin{equation}\label{phiboundsum}
\sum_{k;(j_1,j_2,j_3)\in NR_k} |\varphi_{k,j_1,j_2,j_3}|N_{j_1}N_{j_2}N_{j_3}P_k\leq C\|N\|_{l^2}^2 \|N\|_{l^{2,0^+}}\|P\|_{l^2}.
\end{equation}
\end{lemma}

\begin{proof}
By symmetry in the $k,j_1,j_2,j_3$ variables, to get \eqref{phibound} it is enough to get 
$$\varphi_{k,j_1,j_2,j_3}\leq \frac C{\langle k\rangle},\quad \forall  k, (j_1,j_2,j_3)\in NR_k.$$
To get this estimate we shall decompose the summation in $k,j_1,j_2,j_3$ in two regions as follows.\\

{\em{a) Case $|k|\lesssim \max\{  |k-j_1|, |j_1-j_2|\}$}.} As $r<\frac 12$ we have $|(\langle x\rangle^{2r})'|=|2rx\langle x\rangle^{2r-1}|\leq C\langle x\rangle^{2r-1}$ and the mean value theorem (used in two manners: on $\langle k\rangle ^{2r}-\langle j_1\rangle ^{2r}$ and $\langle j_2\rangle ^{2r}-\langle j_3\rangle ^{2r}$ to get $|j_1-j_2|$-decay, as well as on $\langle k\rangle ^{2r}-\langle j_3\rangle ^{2r}$ and $\langle j_1\rangle ^{2r}-\langle j_2\rangle ^{2r}$ to get $|k-j_1|-$decay) implies that
$$
\varphi_{k,j_1,j_2,j_3}\leq \frac {C}{\max\{|k-j_1|,|j_1-j_2|\}\}}\leq \frac{C}{\langle k\rangle}.
$$

{\em{b) Case $|k|\gg\max\{ |k-j_1|, |j_1-j_2|\}$}.} In this case $|k|\lesssim |j_1|$, $k$ and $j_1$ have same sign, then also $|k|\lesssim |j_2|$, and $j_1$ and $j_2$ have same sign, and we also directly get $|k|\lesssim |j_3|$, and $k$ and $j_3$ have same sign.
As $|(\langle x\rangle^{2r})''|\leq C\langle x\rangle^{2r-2}$, applying twice the mean value theorem, which yields at the numerator  the double integral $\int_{j_1}^k\int_{y}^{y-j_1+j_2}(\langle x\rangle^{2r})''dxdy$, whose all boundary values are of same sign, we get that
$$
\varphi_{k,j_1,j_2,j_3}\leq \frac {C}{ \min\{\langle k\rangle, \langle j_1\rangle,\langle j_2\rangle,\langle j_3\rangle\}^{2-2r}}\leq \frac{C}{\langle k\rangle}.
$$

Estimate \eqref{phiboundsum} follows using \eqref{phibound} and Cauchy-Schwarz:
$$\sum_{k;(j_1,j_2,j_3)\in NR_k} |\varphi_{k,j_1,j_2,j_3}|N_{j_1}N_{j_2}N_{j_3}P_k$$
$$\leq \sum_{k,j_1}\frac{C}{\langle j_1\rangle^{\frac 12^-}\langle k\rangle ^{\frac 12^+}} N_{j_1}P_k \sum_{j_2\neq 0} (N_{j_2}^2+N_{k-j_1+j_2}^2)\leq C\|N\|_{l^2}^2\|N\|_{l^{2,0^+}}\|P\|_{l^2}.$$

\end{proof}
By using Lemma \ref{lemmabound} we can control as wanted in \eqref{neweq} the first three terms in \eqref{diffr}, while for the last term we use the expression of $\partial_\tau B_j$ from \eqref{Bjsyst} to get: 
\begin{equation}\label{neweqbis}
\sum_k \langle k\rangle^{2r}|B_k(t)|^2\leq \sum_k \langle k\rangle^{2r}|B_k(1)|^2+ C(r,\|\{B_j(1)\}\|_{l^2})\sup_{\tau\in [1,t]}\|\{B_j(\tau)\}\|_{l^{2,r}}^{0^+}
\end{equation}
$$+C\sum_{0\leq \nu\leq t}\Big|\int_1^t \sum_{N_j}\varphi_{j_4-j_5+j_6,j_1,j_2,j_3}e^{i\tau  (j_1^2-j_2^2+j_3^2-j_4^2+j_5^2-j_6^2)}B_{j_1}\overline{B_{j_2}}B_{j_3}\overline{B_{j_4}}B_{j_5}\overline{B_{j_6}}(\tau)\frac{d\tau}{\tau^2}\Big|,$$
where $$N_j:=\{(j_1,j_2,j_3,j_4,j_5,j_6),\, j_1-j_2+j_3-j_4+j_5-j_6=0,\, j_1^2-j_2^2+j_3^2-(j_4-j_5+j_6)^2\neq 0\}.$$ 
We are left with the last term, that by using the notation $v(\tau,x)=\sum_j e^{ijx+i\tau j^2}B_j(\tau)$ from \S 2.1, the partition of unity in time of $[1,\infty)$ from \S 2.2 with $\eta_{\nu}$ supported in $(\nu-1,\nu+2)$ and a $\nu_1$ supported away from zero, and the notation $v_\nu$ for a smooth in time restriction to $(\nu-1,\nu+2)$, can be rewritten as
$$C\sum_{0\leq \nu\leq t}\Big|\int_1^t \sum_{N_j}\varphi_{j_4-j_5+j_6,j_1,j_2,j_3}\eta_\nu(\tau) \widehat{v_\nu(\tau)}(j_1)\overline{\widehat{v_\nu(\tau)}(j_2)}\widehat{v_\nu(\tau)}(j_3)\overline{\widehat{v_\nu(\tau)}(j_4)}\widehat{v_\nu(\tau)}(j_5)\overline{\widehat{v_\nu(\tau)}(j_6)}\frac{d\tau}{\tau^2}\Big|.$$
We use now  also Fourier transform in time to write the terms in the summation in $\nu$ as:
$$I_\nu:=\Big| \frac 1{\nu^2} \int_{N_\lambda}\sum_{N_j}\varphi_{j_4-j_5+j_6,j_1,j_2,j_3} \widehat {\mu_\nu v_{\nu}}(\lambda_1,j_1)\overline{\hat v_{\nu}(\lambda_2,j_2)} \hat v_{\nu}(\lambda_3,j_3)\overline{\hat v_{\nu}(\lambda_4,j_4)} \hat v_{\nu}(\lambda_5,j_5)\overline{\hat v_{\nu}(\lambda_6,j_6)} d\Lambda\Big|,$$
where $\mu_\nu(\tau):=\frac{\nu^2}{\tau^2}\eta_\nu(\tau)\mathbb I_{[1,t]}(\tau)$, $d\Lambda:=d\lambda_1...d\lambda_6$ and $N_\lambda$ is the restriction set to $\lambda_1-\lambda_2+\lambda_3-\lambda_4+\lambda_5-\lambda_6=0$. 

We collect now some basic results about the spaces $X^{s,b}$. One can easily check that 
\begin{equation}\label{ldgjfdslkjhdgks} 
 \| \frac{\nu^2}{\tau^2}\eta_\nu(\tau) F(\tau,x) \|_{X^{s, b}} \lesssim 
\| F \|_{X^{s, b}}.
\end{equation}
We have (Lemma 2.1 in \cite{So}), for all $\varepsilon >0$ sufficiently small and $b < 1/2$:
\begin{equation}\label{Bouest}
\| 1_{[\alpha,\beta]} (\tau) F(\tau,x) \|_{X^{s, b}} \lesssim_{\varepsilon} 
\| F \|_{X^{s, b + \varepsilon}},
\end{equation}
with constant that does not depend on the interval $[\alpha,\beta]$.
Interpolating the embeddings $X^{0, 3/8} \hookrightarrow L^{4}_{\tau,x}$
and $X^{0+, \frac12 + } \hookrightarrow L^{6}_{\tau,x}$ (see for instance (27)-(28) in \cite{So}), for $\varepsilon >0$ yields
\begin{equation}\label{Bouest2}
\| F \|_{L^{6-\varepsilon'}_{\tau,x}} \lesssim_{\varepsilon',  \varepsilon'' } 
\| F \|_{X^{\varepsilon'', \frac12 - 2\varepsilon''}}  ,
\end{equation}
as long as we take $0 < \varepsilon'' < \varepsilon_1 < \varepsilon' <  \varepsilon_2$ with 
$\varepsilon_1$ and $\varepsilon_2$ sufficiently small thresholds.

Getting back to estimating $I_\nu$ we distinguish two cases.
First we assume that 
\begin{equation}\label{Cond1}
\max \{\langle  j_4\rangle, \langle j_5\rangle, \langle j_6\rangle \} \gtrsim   \max \{\langle j_4 - j_5 +j_6\rangle, \langle j_1\rangle, \langle j_2\rangle, \langle j_3\rangle \},
\end{equation}
We assume, without loss of generality, that the max is attained at $\langle j_4\rangle$ and also at $\langle j_5\rangle$, since the alternated sum of the $j$'s vanishes. It will be clear by the argument that all the other 
combination of indexes are treated in the exact same way. 
We then use the estimate
$$
| \varphi_{j_4 - j_5 +j_6,j_1,j_2,j_3} | \lesssim  1, 
$$ 
that follows from Lemma \ref{lemmabound}, to bound
\begin{align}\label{nfdjlskjnfgs}
I_\nu \leq  \frac{C}{\nu^2}   
\int_{N_\lambda}
 \sum_{N_j} |\widehat {\mu_\nu v_{\nu}}(\lambda_1,j_1)|\,\Pi_{l=2}^6|\hat v_{\nu}(\lambda_l,j_l)| d\Lambda.
\end{align}
Defining for any $x\in\mathbb T$ and  $6 \varepsilon \ll r$:
$$
F_1 (\tau,x) :=  \int \sum_{j_1}e^{i x j_1+i \tau \lambda_1  } 
\langle j_1\rangle^{-\varepsilon}  |\widehat{\mu_\nu v_\nu}(j_1, \lambda_1)| d \lambda_1,\, F_4 (\tau, x) := \int \sum_{j_4}  e^{i x j_4 + i\tau \lambda_4  } 
\langle j_4\rangle ^{5\varepsilon}   |\widehat{ v_\nu}(j_4, \lambda_4)| d \lambda_4,
$$
\begin{equation}\label{kdjgfklsekjgkdlsm}
F_j (\tau,x) :=\int  \sum_{j_l}  e^{i xj_l+i \tau \lambda_j  } 
\langle j_l\rangle^{-\varepsilon} |\widehat{v_\nu}(k_j, \lambda_j)| d \lambda_j, \qquad l = 2, 3, 5, 6,
\end{equation}
from \eqref{nfdjlskjnfgs} we obtain by using the vanishing of the alternated sums of the $j$'s and $\lambda$'s that 
\begin{align}
I_\nu \leq  \frac{C}{\nu^2}  \int \int F_1\overline{F_2}F_3\overline{F_4}F_5\overline{F_6}(\tau,x)   dx d\tau.
\end{align}
Taking suitable $\varepsilon''\ll \varepsilon' \ll \varepsilon$ and using
\eqref{Bouest2}, \eqref{Bouest} and \eqref{ldgjfdslkjhdgks} we note that 
$$
\| F_1 \|_{L^{6-\varepsilon'}_{\tau,x}} \leq C\| F_1 \|_{X^{\varepsilon'', \frac12 - 2\varepsilon''}} \leq C  \| \mu_\nu v_\nu \|_{X^{\varepsilon''-\varepsilon, \frac12 - 2\varepsilon''}}\leq C  \| \frac {\nu^2} {\tau^2}\eta_\nu v_\nu(\tau,x) \|_{X^{\varepsilon''-\varepsilon, \frac12 - \varepsilon''}}$$
$$\leq C  \| v \|_{X_\nu^{\varepsilon''-\varepsilon, \frac12 - \varepsilon''}}\leq C  \| v\|_{X^{0,b}_\nu},
$$
as $b \in (\frac12,\frac 58)$. Thus using
the H\"older inequality, the embedding $X^{0+, \frac12 +} \hookrightarrow L^{6}_{\tau,x}$ and its consequence by Sobolev embeddings
$X^{\frac{\delta}{2(6+\delta)}+, \frac12 +} \hookrightarrow L^{6+\delta }_{\tau,x}$ (see for instance Corolary 2.1 of \cite{Gr}) we get:
\begin{align}
I_\nu&  \leq  \frac{C}{\nu^2} \Big( \| F_1 \|_{L^{6-\varepsilon'}_{\tau,x}} 
\| F_4 \|_{L^{6+\frac{3\varepsilon'}{3-\varepsilon'}}_{\tau,x}}  \prod_{j=2,3,5,6} \| F_j \|_{L^{6}_{\tau,x}}  \Big)\leq \frac{C}{\nu^2} \Big( \| v \|_{X^{6 \varepsilon, b}_\nu} 
 \| v \|^5_{X^{0, b}_\nu}   \Big)
\\ \nonumber &
 \leq
\frac{C}{\nu^2}  
\| v(\nu) \|_{H^{6 \varepsilon}} \| v(\nu) \|^5_{L^2}\leq
\frac{C\|\{B_j(1)\}\|_{l^2}^5}{\nu^2}\sup_{\tau\in [1,t]}\|\{B_j(\tau)\}\|_{l^{2,6 \varepsilon}}
\\ \nonumber &
 \leq
\frac{C(r,\|\{B_j(1)\}\|_{l^2})}{\nu^2}\sup_{\tau\in [1,t]}\|\{B_j(\tau)\}\|_{l^{2,r}}^{0^+},
\end{align} 
as $6 \varepsilon \ll r$. Therefore we get \eqref{neweq} in the case \eqref{Cond1}.

In the remaining case when \eqref{Cond1} does not hold we have 
$$
\max \{\langle  j_4\rangle, \langle j_5\rangle, \langle j_6\rangle \} \lesssim  \max \{\langle j_4 - j_5 +j_6\rangle, \langle j_1\rangle, \langle j_2\rangle, \langle j_3\rangle \}.
$$
Thus we may assume without loss of generality that the maximum is attained on 
$\langle j_1\rangle$. 
We also recall that from Lemma \ref{lemmabound} we get 
\begin{equation}\label{mfkdlsfgmdksldkgj}
| \varphi_{j_4 - j_5 +j_6,j_1,j_2,j_3} | \lesssim 
\frac{1}{\max \{\langle j_4 - j_5 +j_6\rangle, \langle j_1\rangle, \langle j_2\rangle, \langle j_3\rangle \}} .
\end{equation}
We define:
$$
F_1 ( \tau,x) := \sum_{k_1} \int e^{i xj_1  + i\tau \lambda_1  } 
\langle j_1\rangle^{5 \varepsilon -1} |\widehat{\mu_\nu v}(j_1, \lambda_1)| d \lambda_1,
$$
and $F_j$ as in \eqref{kdjgfklsekjgkdlsm} but for $j=2,3,4,5,6$. Proceeding as above (using the $L^{6+\varepsilon'}$ norm for $F_2$, for instance) we arrive to the estimate
\begin{align}
I_\nu \leq    \frac{C}{\nu^2}  \| v \|_{X^{ 6 \varepsilon -1, b}_\nu} 
 \| v \|^5_{X^{0, b}_\nu}   
 \leq  \frac{C}{\nu^2}  
\| v \|^6_{X^{0, b}_\nu} 
 \leq  \frac{C}{\nu^2}   \| v(\nu) \|^6_{L^2}
 \leq   C  \| v(1) \|^5_{L^2}=C(\|\{B_j(1)\}\|_{l^2}),  
\end{align} 
where here we choose $\varepsilon<\frac 16$. This is again sufficient for getting \eqref{neweq}.

\end{proof}


\subsection{The asymptotic behavior of $B_j$'s}

We shall prove now first the following asymptotic behavior.
\begin{lemma}\label{propas}
For data $\{B_k(1)\}\in l^{2,s}$ with $0<s<\frac 12$ and for all $k\in\mathbb Z$ we have:
$$\exists\,\beta_k:=\lim_{t\rightarrow\infty}\tilde B_k(t),$$
where
$$\tilde B_k(t):=e^{i2M\log t-i\int_1^t|B_k(\tau)|^2\frac{d\tau}{\tau} }B_k(t),$$
with the following estimate on the rate of convergence:
$$\sup_{k\in\mathbb Z}|\tilde B_k(t)-\beta_k|\leq  \frac{C(\|\{B_j(1)\}\|_{ l^{2,s}})}{t}.$$
\end{lemma}

\begin{proof}
We first note that in view of \eqref{Bjsystdev} we get
$$i\partial_t \tilde B_k(t)=\frac{e^{i2M\log t-i\int_1^t|B_k(\tau)|^2\frac{d\tau}{\tau} }}{t}\sum_{NR_k}e^{-it\omega_{k,j_1,j_2}}B_{j_1}(t)\overline{B_{j_2}(t)}B_{j_3}(t),$$
so for $t_2\geq t_1\geq 1$ we have by integration by parts:
$$\tilde B_k(t_2)-\tilde B_k(t_1)=\int_{t_1}^{t_2}\frac{1}{\tau}\sum_{NR_k}e^{-i\tau\omega_{k,j_1,j_2}}e^{i2M\log \tau-i\int_1^\tau |B_k(\tau')|^2\frac{d\tau'}{\tau'} }(B_{j_1}\overline{B_{j_2}}B_{j_3})(\tau)d\tau$$
$$=\Big[\sum_{NR_k}\frac{e^{-i\tau\omega_{k,j_1,j_2}}}{\omega_{k,j_1,j_2}}\frac{e^{i2M\log \tau-i\int_1^\tau |B_k(\tau')|^2\frac{d\tau'}{\tau'}}}{\tau}(B_{j_1}\overline{B_{j_2}}B_{j_3})(\tau)\Big]_{t_1}^{t_2}$$
$$-\int_{t_1}^{t_2}\sum_{NR_k}\frac{e^{-i\tau\omega_{k,j_1,j_2}}}{\omega_{k,j_1,j_2}}\Big(\frac{e^{i2M\log \tau-i\int_1^\tau |B_k(\tau')|^2\frac{d\tau'}{\tau'}}}{\tau}(B_{j_1}\overline{B_{j_2}}B_{j_3})(\tau)\Big)_\tau d\tau.$$
The boundary term can be easily upper-bounded by Cauchy-Schwarz in $j_1,j_2$ by
$$\sup_{\tau\in [t_1,t_2]}\frac 1\tau \sum_{NR_k}\frac{|B_{j_1}B_{j_2}B_{j_3}|)(\tau)}{|(k-j_1)(j_1-j_2)|}\leq \frac{C(\|\{B_j(1)\}\|_{l^2})}{t_1}.$$
So can be the terms in the integral when the derivative in $\tau$ falls on $\frac{e^{i2M\log \tau-i\int_1^\tau |B_k(\tau')|^2\frac{d\tau'}{\tau'}}}{\tau}$. Remains to estimate just one kind of term, when the derivative falls on $B_{j_3}$ for instance, on which we use the partition of unity in time:
$$\sum_\nu \frac 1\nu  \int_{t_1}^{t_2}\tilde \eta_\nu(\tau)\sum_{NR_k}\frac{e^{-i\tau\omega_{k,j_1,j_2}}}{(k-j_1)(j_1-j_2)}e^{i2M\log \tau-i\int_1^\tau |B_k(\tau')|^2\frac{d\tau'}{\tau'}}B_{j_1}\overline{B_{j_2}}\partial_\tau B_{k-j_1+j_2} (\tau)d\tau.$$
We use Fourier transform in time to write this term as:
$$J_k:=\sum_{t_1\leq \nu\leq t_2}\Big| \frac 1\nu \int\int\int\int \sum_{m\neq 0;\, j_1,j_2\in\Lambda_{k,m}}\frac {e^{i\tau (m-\lambda_1+\lambda_2-\lambda_3)}}me^{i2M\log \tau-i\int_1^\tau |B_k(\tau')|^2\frac{d\tau'}{\tau'}}\tilde \eta_\nu(\tau)$$
$$\left.\times \hat B_{j_1,\nu}(\lambda_1)\widehat {\overline{B}}_{j_2,\nu}(\lambda_2)\widehat {\partial_\tau B}_{k-j_1+j_2,\nu}(\lambda_3) d\lambda_1d\lambda_2d\lambda_3 d\tau\right|.$$
Now we use Lemma \ref{bounddert}:
$$J_k\leq \sum_{t_1\leq \nu\leq t_2} \frac C{\nu^2} \int\int\int \sum_{m\neq 0;\,j_1,j_2\in\Lambda_{k,m}}\frac{|\widehat{\tilde \eta_{\nu,k}}(m-\lambda_1+\lambda_2-\lambda_3)|}{|m|}|\hat B_{j_1,\nu}(\lambda_1)\widehat {\overline{B}}_{j_2,\nu}(\lambda_2)| d\lambda_1d\lambda_2d\lambda_3$$
$$\leq  \sum_{t_1\leq \nu\leq t_2} \frac {\tilde C}{\nu^2} \int\int\sum_{m\neq 0;\,j_1,j_2\in\Lambda_{k,m}}\frac 1{|m|}|B_{j_1,\nu}(\lambda_1)\widehat {\overline{B}}_{j_2,\nu}(\lambda_2)| d\lambda_1d\lambda_2.$$
Here the constant is the one of Lemma \ref{bounddert} (note that for the sequence of functions $\eta_\nu$ that we are considering we have $\||\xi|^{1^+}\hat\eta_\nu(\xi)\|_{L^\infty}$ uniformly bounded in $\nu$):
\begin{equation}\label{Constagain}
\tilde C=C(\epsilon,\|\{B_j(1)\}\|_{ l^{2}}) \|\{B_j(\nu)\}\|_{l^{2,\epsilon}},
\end{equation}
Now we perform successively Cauchy-Schwarz in $j_1,j_2$, for which we use again, for $m\neq 0$ and $\epsilon>0$:
$$\# \Lambda_{k,m}\leq C(\epsilon)m^\epsilon,$$ 
then in $m$, and eventually in $\lambda_1,\lambda_2$:
$$J_k\leq \tilde C\sum_{t_1\leq \nu\leq t_2} \frac { C(\epsilon)}{\nu^2} \int\int\sum_{m\neq 0}\frac{1}{|m|^{1-\epsilon}}\Big(\sum_{j_1,j_2\in\Lambda_{k,m}}  |\hat B_{j_1,\nu}(\lambda_1)\widehat {\overline{B}}_{j_2,\nu}(\lambda_2)|^2\Big)^\frac 12 d\lambda_1d\lambda_2$$
$$\leq \tilde C\sum_{t_1\leq \nu\leq t_2} \frac { C(\epsilon)}{\nu^2} \int\int\Big(\sum_{m\neq 0;\,j_1,j_2\in\Lambda_{k,m}} \frac{1}{|m|^{1-4\epsilon}}  |\hat B_{j_1,\nu}(\lambda_1)\widehat {\overline{B}}_{j_2,\nu}(\lambda_2)|^2\Big)^\frac 12 d\lambda_1d\lambda_2$$
$$\leq \tilde C\sum_{t_1\leq \nu\leq t_2} \frac { C(\epsilon)}{\nu^2} \Big(\int\int\sum_{j_1,j_2} |\langle \lambda_1\rangle^b\hat B_{j_1,\nu}(\lambda_1)\langle \lambda_2\rangle^b\widehat {\overline{B}}_{j_2,\nu}(\lambda_2)|^2d\lambda_1d\lambda_2\Big)^\frac 12 $$
$$\leq \tilde C\sum_{t_1\leq \nu\leq t_2} \frac { C(\epsilon)}{\nu^2} \|B\|_{H^{b,0}_\nu}^2\leq  \tilde C\sum_{t_1\leq \nu\leq t_2} \frac { C(\epsilon,\|\{B_j(1)\}\|_{ L^{2}})}{\nu^2}  \leq \frac{C(\epsilon,\|\{B_j(1)\}\|_{ l^{2}})}{t_1}\sup_{\tau\in[t_1,t_2]}\|\{B_j(\tau)\}\|_{l^{2,\epsilon}}$$
$$\leq \frac{C(\epsilon,\|\{B_j(1)\}\|_{ l^{2}})}{t_1}\|\{B_j(1)\}\|_{l^{2,\epsilon}}\leq \frac{C(\epsilon,\|\{B_j(1)\}\|_{ l^{2}})}{t_1}\|\{B_j(1)\}\|_{l^{2,s}}^{0^+},$$
where we have used the expression of the constant from \eqref{Constagain}, Lemma \ref{lemmawev} and interpolation.

\end{proof}

We can now prove Theorem \ref{wbl}. In view of \eqref{Bjsystdev} we have
$$i\partial_t \Big(e^{i(2M-|\beta_k|^2)\log t }B_k(t)\Big)=\frac{e^{i(2M-|\beta_k|^2)\log t }}{t}\sum_{NR_k}e^{-it\omega_{k,j_1,j_2}}B_{j_1}(t)\overline{B_{j_2}(t)}B_{j_3}(t)$$
$$-\frac {e^{i(2M-|\beta_k|^2)\log t }}t (|B_k(t)|^2-|\beta_k|^2)B_k(t).$$
Then, by integrating in time the first term as for the proof of Lemma \ref{propas}, and by integrating the second term by using the convergence result of Lemma \ref{propas} we obtain
$$\exists\,\alpha_k:=\lim_{t\rightarrow\infty}e^{i(2M-|\beta_k|^2)\log t }B_k(t),$$
with the following estimate on the rate of convergence:
$$|B_k(t)-\alpha_ke^{i(2M-|\alpha_k|^2)\log t }|\leq  \frac{C(\|\{B_j(1)\}\|_{ l^{2}})}{t}.$$
Recalling definition \eqref{mass} of the ``mass" we obtain the conclusion \eqref{blupth} of Theorem \ref{wbl}.\\

\begin{remark}
Theorem \ref{wbl} corresponds to asymptotic completeness for the $B_j$'s. For proving existence of wave operators at the same level of regularity one may proceed as follows. 
Let $\{\alpha_k\}_{k\in\mathbb Z}\in l^{2,s}$ with $0<s$ and let $M:=\sum_k|\alpha_k|^2$ and construct solutions for $t\geq T(\{\alpha_k\})$ to 
$$i\partial_t B_{k}=\frac 1t \sum_{NR_k}e^{it\omega_{k,j_1,j_2}}(\alpha_{j_1}+B_{j_1})\overline{(\alpha_{j_2}+B_{j_2})}(\alpha_{j_3}+B_{j_3})(t)-\frac 1t|(\alpha_k+B_{k})|^2(\alpha_k+B_{k})(t),$$
with $\|B_k(t)\|_{l^{2,s}}\overset{t\rightarrow\infty}{\longrightarrow} 0$ by performing a logarithmic phase change and by doing a fixed point argument for\footnote{To avoid smallness hypothesis on $\{\alpha_k\}_{k\in\mathbb Z}$ that would come from the linear term in the last integral term of the fixed point operator, this last integrant can be replaced by an oscilatory one in the same spirit as in \S 2.1 of \cite{BVAnnPDE}.} 
$$\Phi(\{B_{j}\}_{j\in\mathbb Z})_k(t)=i\int_t^\infty  \sum_{NR_k}e^{i\tau\omega_k+i(|\alpha_{j_1}|^2-|\alpha_{j_2}|^2+|\alpha_{j_3}|^2-|\alpha_k|^2)\log\tau}(\alpha_{j_1}+B_{j_1})\overline{(\alpha_{j_2}+B_{j_2})}(\alpha_{j_3}+B_{j_3})(\tau)\frac{d\tau}\tau,$$
$$-i\int_t^\infty (|(\alpha_k+B_{k})|^2-|\alpha_k|^2)(\alpha_k+B_{k})(\tau)\frac{d\tau}\tau,$$
in the space of sequences $\{B_{k}\}_{k\in\mathbb Z}$ having the property that for all truncations $\eta_\nu$ of support of size one, at distance of size $ \nu$ from the origin, 
$$\sum_{\lambda,k} \langle \lambda\rangle^{2b}\langle k\rangle^{2s} |\widehat{\eta_\nu B_k}(\lambda)|^2\leq c(\nu),$$ 
with $c(\nu)$ summable in $\nu$ and $b>\frac 12$.
\end{remark}
\medskip


\section{Proof of Theorem \ref{thbf}}
As explained in the Introduction, we shall prove that $\|\{B_j(t)\}\|_{l^{2,\frac 32^+}}$ is uniformly bounded in time. To do so we shall first prove a uniform in time bound on the $l^{2,1}$-norms, then we shall revisit Lemma \ref{bounddert} at higher regularity, and gathering all together we shall estimate the evolution in time of the $l^{2,\frac 32^+}$-norm similarly as for Lemma \ref{lemmawev}.

We start by noting that from \eqref{NLSv} we get the energy evolution law:
If
\begin{equation}\label{E}E(v)(t):=\frac12\|\nabla v(t)\|_{L^2}^2-\frac{1}{4t}\| v(t) \|_{L^4}^4,
\end{equation}
then
\begin{equation}\label{Ederivative}
\partial_tE(v)(t)=\frac 1{4t^2} \| v(t) \|_{L^4}^4.
\end{equation}

\begin{lemma}\label{controlenergy}
For $0<s<\frac 12$ we have the following control:
$$\sum_k \langle k\rangle^{2}|B_k(t)|^2\leq \sum_k \langle k\rangle^{2}|B_k(1)|^2+ C(s,\|\{B_j(1)\}\|_{l^2})\|\{B_j(1)\}\|_{l^{2,s}}^\theta,$$
where $\theta(s)>2$ and $\theta(\frac 12^-)=2^+$.
\end{lemma}
\begin{proof}
By using the energy law \eqref{E}-\eqref{Ederivative} and Sobolev embeddings we have
$$\|\nabla v(t)\|_{L^2}^2-\|\nabla v(1)\|_{L^2}^2\leq C\sup_{1\leq \tau\leq t}\|v(\tau)\|_{L^4}^4\leq C \sup_{1\leq \tau\leq t}\|v(\tau)\|_{H^\frac 14}^4=C \sup_{1\leq \tau\leq t}\|\{B_j(\tau)\}\|_{H^{\frac 14}}^4.$$
Therefore by using Lemma \ref{lemmawev} with $r=\frac 14$ we end up with $\|\{B_j(1)\}\|_{l^{2,\frac 14}}^4$ that gives by interpolation $\|\{B_j(1)\}\|_{l^{2,s}}^\theta(s)$ with $\theta(s)>2$ as $s<\frac 12$, and $\theta(\frac 12^-)=2^+$.

\end{proof}

\begin{lemma}\label{lemmawevbis}
We have:
$$\|\{B_j(t)\}\|_{l^{2,\frac 32^+}}\leq C( r,\|\{B_j(1)\}\|_{ l^{2,\frac 32^+}}).$$
\end{lemma}

\begin{proof}
The proof follows the one of Lemma \ref{lemmawev} with $r=\frac 32^+$, which gives in particular the weaker estimate (see Lemma 2.5 in \cite{FoSe}):
$$\Big|\frac{\langle k\rangle ^{2r}-\langle j_1\rangle ^{2r}+\langle j_2\rangle ^{2r}-\langle j_3\rangle ^{2r}}{k^2-j_1^2+j_2^2-j_3^2}\Big|\leq C\max\{\langle k\rangle,\langle j_1\rangle,\langle j_2\rangle,\langle j_3\rangle\}^{1^+},\quad \forall  k, (j_1,j_2,j_3)\in NR_k.$$
On the first three terms in \eqref{diffr} we use this estimate and Lemma \ref{controlenergy} to get
$$\sum_{k;NR_k}|\varphi_{k,j_1,j_2,j_3}B_{j_1}\overline{B_{j_2}}B_{j_3}\overline{B_{k}(\tau)}|\leq C\|B(\tau)\|_{l^{2,1}}^4\leq C(\|\{B_j(1)\}\|_{ l^{2,1}}).$$
The last integral term in \eqref{diffr} involving $\partial_\tau B_k$ can be treated similarly as in the beginning of the proof of Lemma \ref{lemmawev},  distributing the loss $\max\{\langle j_4-j_5+j_6\rangle,\langle j_1\rangle,\langle j_2\rangle,\langle j_3\rangle\}^{1^+}$ coming from $\varphi_{j_4-j_5+j_6,j_1,j_2,j_3}$ on the terms corresponding to the highest frequencies: in the case \eqref{Cond1} on $F_4$ and $F_5$, and in the remaining case on $F_1$. 
\end{proof}

\end{document}